\documentclass{amsart}

\usepackage{lineno}
\usepackage{marginnote,soul}
\usepackage{Khanh}
\usepackage{float}
\newtheorem{thm}{Theorem}[section]

\newtheorem{lem}[thm]{Lemma}
\newtheorem{cor}[thm]{Corollary}

\newtheorem{remark}[thm]{Remark}

\makeatletter
\renewenvironment{proof}[1][\proofname]{%
   \par\pushQED{\qed}\normalfont%
   \topsep6\p@\@plus6\p@\relax
   \trivlist\item[\hskip\labelsep\bfseries#1\@addpunct{.}]%
   \ignorespaces
}{%
   \popQED\endtrivlist\@endpefalse
}

\definecolor{applegreen}{rgb}{0.55, 0.71, 0.0}

\usepackage[backend=biber,style=alphabetic,sorting=nyt,
    url=false,doi=false,isbn=false,
]{biblatex}
\usepackage{subfig}

\title{A new twist on modular links from an old perspective}
\author{Khanh Le}
\begin{document}
\maketitle
\begin{abstract}
     We show that the complement of arithmetic modular links found in \cite{MiguelesPinskyPurcell} is homeomorphic to the complement of augmented chainlinks. In particular, these link complements arise as $n$-fold cyclic covers of the Whitehead link complement.  
\end{abstract}

\section{Introduction}
    The \emph{modular surface} $\Sigma_{\Mod}$ is an orbifold obtained as the quotient space of the hyperbolic plane $\mathbb{H}^2$ by the modular group $\PSL_2(\mathbb{Z})$. Since the action of $\PSL_2(\mathbb{Z})$ on $\mathbb{H}^2$ is by orientation-preserving isometries, $\Sigma_{\Mod}$ is an oriented 2-orbifold equipped with a hyperbolic metric. Any closed oriented geodesic $\overline{\gamma}(t)$ on $\Sigma_{\Mod}$ has a canonical lift $\gamma(t) := (\overline{\gamma}(t),\overline{\gamma}'(t))$ to the unit tangent bundle $\UT(\Sigma_{\Mod})$. Milnor showed that $\UT(\Sigma_{\Mod})$ is homeomorphic to the complement of the trefoil knot $T_{2,3}$ in $S^3$ \cite{MilnorKTheory}. Therefore, every nonempty finite collection of canonical lifts $ \Gamma \subset \UT(\Sigma_{\Mod})$ of oriented closed geodesics in $\Sigma_{\Mod}$ together with the trefoil knot determines a $n+1$-component link $\Gamma\cup \{T_{2,3}\}$ in $S^3$ for $n\geq 1$. Following Ghys \cite{GhysICM06}, we refer to the collection $\Gamma$, without the trefoil knot, as a \emph{modular link} when $|\Gamma| \geq 2$ and \emph{modular knot} when $|\Gamma| = 1$. Here $|\cdot|$ denotes the number of connected components. The \emph{complement of modular links} refers to $M_{\Gamma} := \UT(\Sigma_{\Mod})\setminus \Gamma$. For emphasis, the complement of the modular link $\Gamma$ is the complement of $(n+1)$ -- component link in $S^3$ where $|\Gamma| = n$ is the number of components of the modular link.

    
    Modular links have attracted attention of mathematicians due to their connections to dynamics, low-dimensional topology and number theory. For example, in \cite{GhysICM06}, Ghys showed that the isotopy classes of modular knots coincide with the isotopy classes of Lorenz knots which are periodic orbits of a 3-dimensional differential equation \cite{BirmanWilliams}. Furthermore, Ghys proved that the linking number in $S^3$ between the canonical lift $\gamma$ to $\UT(\Sigma_{\Mod})$ of an oriented closed geodesic $\overline{\gamma}$ in $\Sigma_{\Mod}$ and the trefoil knot $T_{2,3}$ is given by the Rademacher function, a classical arithmetic function coming from number theory \cite{GhysICM06}. The latter result has been generalized to the setting of arbitrary $(p,q,\infty)$-triangle group in \cite{MatsusakaUeki23}. 
    
    The complement of modular links $M_\Gamma $ is known to be hyperbolic \cite{FoulonHasselBlatt13}. More recently, there have been many works relating the hyperbolic volume to the length of the geodesics \cite{BergeronPinskySilberman, CremaschiMigueles, CremaschiMiguelesYarmola, Migueles}. Recently, Migueles, Pinsky and Purcell \cite{MiguelesPinskyPurcell} found an infinite family $\mathcal{F}\Sigma_{\Mod}$ of modular links whose complement $M_\Gamma$ admits an arithmetic hyperbolic structure \cite[Theorem 1.1]{MiguelesPinskyPurcell}. For any $n \geq 3$, the family $\mathcal{F}\Sigma_{\Mod}$ contains at least two modular links with $n$-component. See \cref{subsec:MPPConstruction} for a precise parametrization of modular links in $\mathcal{F}\Sigma_{\Mod}$ in terms of the Farey graph. The hyperbolic structures of $M_\Gamma$ for any collection $\Gamma$ in $\mathcal{F}\Sigma_{\Mod}$ are all commensurable to that of the Bianchi orbifold $\mathbb{H}^3/\PSL_2(\mathbb{Z}[i])$. Furthermore, there is a unique modular knot $\Gamma_0$ in the family $\mathcal{F}\Sigma_{\Mod}$. The complement $M_{\Gamma_0}$ is known to be homeomorphic to that of the Whitehead link \cite{MiguelesPinskyPurcell}. In general, it is an open question that $\Gamma_0$ is the only arithmetic modular knot \cite{MiguelesPinskyPurcell}.  

    The main result of this paper is to explicitly identify the complements of modular links in $\mathcal{F}\Sigma_{\Mod}$ as the complements of augmented chainlinks in $S^3$. These \emph{augmented chain link complements} $S^3\setminus C_{n}$ can be obtained by taking the $n$-fold cyclic cover branched over the unknotted component of the Whitehead link $S^3\setminus C_1$, see \cref{fig:WHL-AugmentedChainLink}. In particular, $C_n$ is a link in $S^3$ with $n+1$ components.

      \begin{figure}
        \centering
        \includegraphics[scale = 0.3]{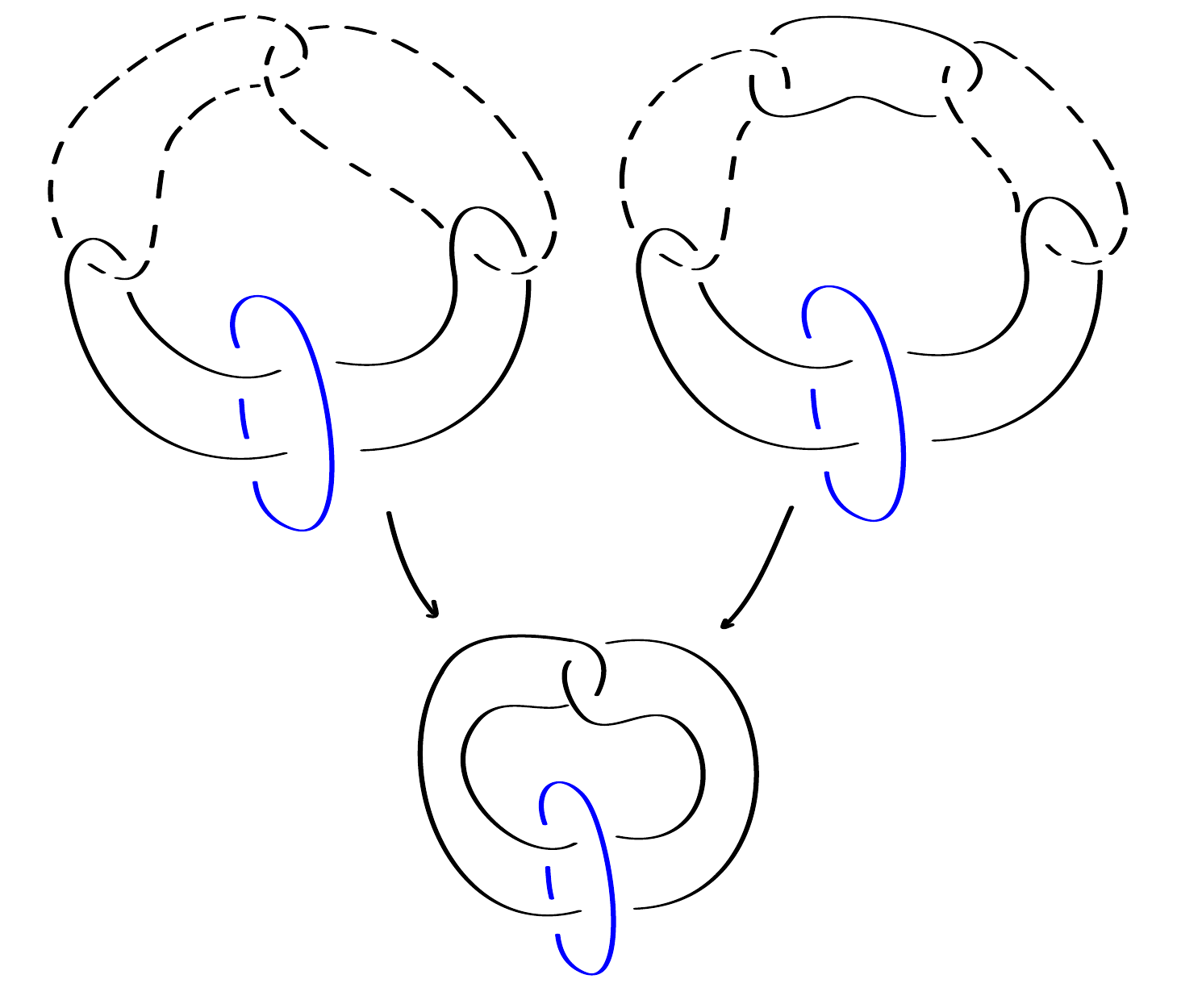}
        \caption{The Whitehead link $S^3\setminus C_1$, a 3-fold cyclic cover $S^3 \setminus C_3$ (top left) and a 4-fold cyclic cover $S^3\setminus C_4$ (top right) both branched over the blue component. Forgetting the dotted components in $S^3\setminus C_3$ and $S^3 \setminus C_4$, we obtain two split links whose complement are handlebodies of genus 2 and 3 respectively. Consequently, we obtain surjective homomorphisms $\pi_1(S^3\setminus C_3) \rightarrow F_2$ and $\pi_1(S^3\setminus C_4) \to F_3$ in each case.}
        \label{fig:WHL-AugmentedChainLink}
    \end{figure}  
    
    \begin{thm}
    \label{thm:ComplementArithmeticModularLinks}
        Let $\Gamma \in \mathcal{F}\Sigma_{\Mod}$ be an $n$-component modular link. The complement $M_\Gamma$ is homeomorphic to the complement $S^3 \setminus C_{n}$.
    \end{thm}


    Using the work of Cooper and Long \cite{CooperLong93}, we obtain the following corollary of \cref{thm:ComplementArithmeticModularLinks}:
    \begin{cor}
    \label{cor:MostArithmeticModularLinksAreLarge}
        Let $\Gamma \in \mathcal{F}\Sigma_{\Mod}$. The complement $M_\Gamma$ fibers. Furthermore, if $|\Gamma| \not\in \{1,2,3,5\}$, then the complement $M_\Gamma$ contains a closed embedded essential surface. 
    \end{cor}


The fact that the complement of modular links fibers was shown by Dehornoy in \cite{Dehornoy15}. In fact, Dehornoy proved a much more general fact: the complement of every finite collection of periodic orbits of the geodesic flow on the unit tangent bundle of the triangle orbifold $(p,q,\infty)$ fibers \cite[Corollary 1.5]{Dehornoy15}.

As a final remark, modular knots, without the trefoil component, considered in  \cite{MiguelesPinskyPurcell} are examples of Berge knots, namely the family of knots which lie as simple closed curves on the fiber of the trefoil knot complement. Chainlinks have also played an important role in the study of the topology and geometry of these Berge knots. In particular, Baker gave a surgery description of Berge knots on the fiber of the trefoil knot using chain links \cite[Proposition 3.1]{Baker08}. Using this, he proved that this family of Berge knots contains hyperbolic knots with arbitrary large volume \cite[Theorem 4.1]{Baker08}. As a consequence, there is no surgery description for these Berge knots on a single link in $S^3$ \cite{Baker08}.

\subsection{Acknowledgement} The author would like to thank Alan Reid for helpful conversations, his enthusiasm for the project, and for his comments on an earlier draft of the paper. The author would also like to thank Neil Hoffman and Ken Baker for some helpful conversations regarding \cite{Baker08}. \cref{fig:3CompArithmeticModularLink} is drawn using Snappy \cite{SnapPy}. The author thanks the anonymous referee for pointing out some mistakes in the previous version of the paper and for making suggestions that have improved the exposition. 

\section{Preliminaries}

We begin by stating some definitions, collecting some standard facts about modular links.  

\subsection{Definitions and background}
The \emph{modular surface}, $\Sigma_{\Mod}$, is the quotient space of $\mathbb{H}^2$ by the group $\PSL_2(\mathbb{Z})$. This group is generated by two elliptic isometries: $U$ which rotates about $i$ by an angle of $\pi$, and $V$ which rotates about $\displaystyle \frac{1+i\sqrt{3}}{2}$ by an angle of $2\pi/3$. As an element of $\PSL_2(\mathbb{Z})$, $U$ and $V$ have the form:
\[
U = \pm \begin{pmatrix}
    0 & -1 \\ 1 & 0
\end{pmatrix} \quad V = \pm\begin{pmatrix}
    0 & -1 \\ 1 & -1 
\end{pmatrix}.
\]
A fundamental domain of for the action of $\PSL_2(\mathbb{Z})$ is the triangle with a real vertex at $\displaystyle \frac{1+i\sqrt{3}}{2}$ and two ideal vertices at $0$ and $\infty$, see \cref{fig:ModularDomain}.
\begin{figure}[h!]
    \centering
    \begin{tikzpicture}[scale =2,thick]
\draw[-] (-2,0) -- (2,0);
\draw[-] (0,0) -- (0,2);
\draw[-] (0.5,{sqrt(3)/2}) -- (0.5,2);
\draw[-] (0,0) arc (180:120:1);
\end{tikzpicture}
    \caption{A fundamental domain of $\Sigma_{\Mod}$ in $\mathbb{H}^2$.}
    \label{fig:ModularDomain}
\end{figure}
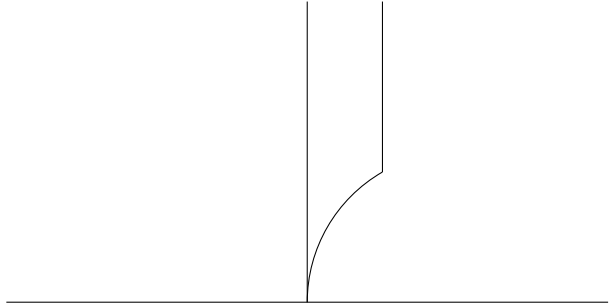
The hyperbolic metric on $\mathbb{H}^2$ descends to a hyperbolic metric on $\Sigma_{\Mod}$ with two points of cone angles $\pi$ and $2\pi/3$ and a single cusp. An oriented simple closed geodesic on $\Sigma_{\Mod}$ corresponds to a conjugacy class of a primitive hyperbolic elements in $\PSL_2(\mathbb{Z})$. Each oriented simple closed geodesic $\gamma$ on $\Sigma_{\Mod}$ has a representative in the corresponding conjugacy class   that admits a factorization into a product of
\[
L = \pm \begin{pmatrix}
    1 & 1 \\ 0 & 1
\end{pmatrix} \quad R = \pm \begin{pmatrix}
    1 & 0 \\ 1 & 1
\end{pmatrix}.
\]
Therefore, we associate to an oriented simple closed geodesic $\gamma$ in $\Sigma_{\Mod}$: a word, $w_\gamma$, in the positive powers of $L$ and $R$ that is not a power of any subword. The correspondence between $\gamma$ and $w_\gamma$ is well-defined up to a cyclic permutation of $w_\gamma$. 

Since $\Sigma_{\Mod}$ comes equipped with a hyperbolic metric, there exists a natural flow on the unit tangent bundle $\UT(\Sigma_{\Mod})$ which is the geodesic flow $\Psi_t$ defined as follows. Given a pair of a point and a unit vector based at the point, $(x,v) \in \UT(\Sigma_{\Mod})$, the geodesic flow moves the point $(x,v)$ in unit speed along the geodesic starting at $x$ tangent to $v$. Each oriented simple geodesic $\gamma$ on $\Sigma_{\Mod}$ has a canonical lift to $\UT(\Sigma_{\Mod})$. The periodic orbits of $\Psi_t$ on $\UT(\Sigma_{\Mod})$ correspond precisely to the canonical lift to $\UT(\Sigma_{\Mod})$ of oriented simple geodesics on $\Sigma_{\Mod}$. 

As noted in the introduction, $\UT(\Sigma_{\Mod})$ is homeomorphic to the complement of the trefoil knot $S^3 \setminus T_{2,3}$. In \cite{GhysICM06}, Ghys showed that periodic orbits of $\Psi_t$ can be isotoped to lie on a branched surface in $S^3 \setminus T_{2,3}$ which is known as the Lorenz template, $\mathcal{T}$, see \cref{fig:modulartemplate}. 

\begin{figure}[h!]
    \centering
    \includegraphics[scale=0.5]{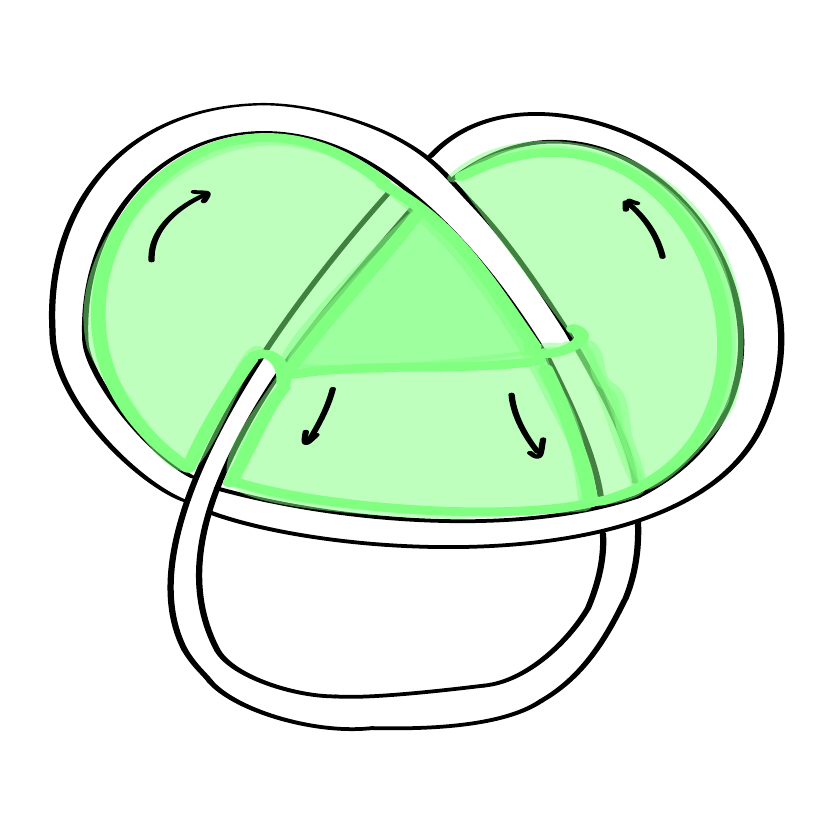}
    \caption{The modular template $\mathcal{T}$ together with a flow}
    \label{fig:modulartemplate}
\end{figure}

The Lorenz template supports a flow which can parametrized as follows. We identify the branching locus of the surface with the open interval $(0,1)$. Starting at any point $x < 1/2$, the flow line follows the left side of the template and returns to the branching locus at the point $2x \mod 1$. If $x > 1/2$, the flow line follows the right side of the template and comes back to the branching locus at the point $2x \mod 1.$ Any periodic orbit of this flow can be determined by a periodic orbit of the times $2$ map on the interval $[0,1]$. Given a sequence of $LR$-word $w_\gamma$, we can obtain the corresponding point in $(0,1)$ by converting $LR$ into a binary sequence by the rule $L\mapsto 0$ and $R \mapsto 1$. Let $\overline{w}_\gamma$ be the decimal number that corresponds to the binary sequence and $|w_\gamma|$ be the length of the $LR$-word. The point in $(0,1)$ that corresponds to $w_\gamma$ is given by
\[
\frac{\overline{w}_\gamma}{2^{|w_\gamma|}-1}
\]
Therefore, given a collection of $LR$-words representing a modular link, we can draw the modular link on the Lorenz template $\mathcal{T}$ by computing the corresponding sequences of periodic orbit on $(0,1)$ and connect them by the flow line on $\mathcal{T}$. For an example of a $3$-component modular link $\{LR,L^2R^2,L^2RLR^2\}$, see \cref{fig:3CompArithmeticModularLink}.


\begin{figure}[h!]
    \centering
    \input{Figures/LR_L2R2_LR2L2R}
    \caption{The 3-component modular link $\{\textcolor{linkcolor1}{LR},\textcolor{linkcolor2}{L^2R^2},\textcolor{linkcolor3}{L^2RLR^2}\}$ and the trefoil knot $\textcolor{linkcolor0}{T_{2,3}}$. The $\textcolor{linkcolor1}{LR}$-component corresponds to the sequence of periodic orbit $\textcolor{linkcolor1}{\{1/3,2/3\}}$ on $(0,1)$. The $\textcolor{linkcolor2}{L^2R^2}$-component corresponds to the sequence of periodic orbit $\textcolor{linkcolor2}{\{1/5,2/5,4/5,3/5\}}$ on $(0,1)$. Finally, the $\textcolor{linkcolor3}{L^2RLR^2}$-component corresponds to the sequence of periodic orbit $\textcolor{linkcolor3}{\{11/63,22/63,44/63,25/63,50/63,37/63\}}$ on $(0,1)$.}
    \label{fig:3CompArithmeticModularLink}
\end{figure}

\subsection{A construction of arithmetic modular links}
\label{subsec:MPPConstruction}
Now we will review the construction of a family of arithmetic modular links $\mathcal{F}\Sigma_{\Mod}$ from \cite{MiguelesPinskyPurcell}. First consider the six-fold cyclic cover of $\Sigma_{\Mod}$ by the once-punctured torus $\Sigma_{1,1}$:
\[
\overline{\pi}:\Sigma_{1,1} \to \Sigma_{\Mod}.
\]
Viewing $\Sigma_{1,1}$ as the quotient $(\mathbb{R}^2 \setminus \mathbb{Z}^2) / \mathbb{Z}^2$, we see that $\Sigma_{1,1}$ can be identified with the square torus with with a point removed, see \cref{fig:OncePuncturedTorus}. A geodesic connecting the cone point of order $2$ and the cusp of $\Sigma_{\Mod}$ lifts to a collection of three cusp-to-cusp geodesics on $\Sigma_{1,1}$. 

\begin{figure}[h!]
    \centering
    \begin{tikzpicture}[scale=2]
        \draw[line width=.5mm] (0,0) -- (-1/2,{sqrt(3)/2}) -- (1/2,{sqrt(3)/2}) -- (1,0) -- (0,0);
        \draw[line width=.5mm] (0,0)  -- (1/2,{sqrt(3)/2});
        
        \draw (0,0) circle[radius=1pt]; and \fill[blue] (0,0) circle[radius=1pt];
        \draw (-1/2,{sqrt(3)/2}) circle[radius=1pt]; and \fill[blue] (-1/2,{sqrt(3)/2}) circle[radius=1pt];
        \draw (1/2,{sqrt(3)/2}) circle[radius=1pt]; and \fill[blue] (1/2,{sqrt(3)/2}) circle[radius=1pt];
        \draw (1,0) circle[radius=1pt]; and \fill[blue] (1,0) circle[radius=1pt];

        \draw[line width=.5mm] (-2,0) -- (-1/2-2,{sqrt(3)/2}) -- (1/2-2,{sqrt(3)/2}) -- (1-2,0) -- (0-2,0);
        \draw[line width=.5mm] (0-2,0)  -- (1/2-2,{sqrt(3)/2});
        
        \draw (0-2,0) circle[radius=1pt]; and \fill[blue] (0-2,0) circle[radius=1pt];
        \draw (-1/2-2,{sqrt(3)/2}) circle[radius=1pt]; and \fill[blue] (-1/2-2,{sqrt(3)/2}) circle[radius=1pt];
        \draw (1/2-2,{sqrt(3)/2}) circle[radius=1pt]; and \fill[blue] (1/2-2,{sqrt(3)/2}) circle[radius=1pt];
        \draw (1-2,0) circle[radius=1pt]; and \fill[blue] (1-2,0) circle[radius=1pt];

        \draw[line width=.5mm] (2,0) -- (-1/2+2,{sqrt(3)/2}) -- (1/2+2,{sqrt(3)/2}) -- (1+2,0) -- (0+2,0);
        \draw[line width=.5mm] (0+2,0)  -- (1/2+2,{sqrt(3)/2});
        
        \draw (0+2,0) circle[radius=1pt]; and \fill[blue] (0+2,0) circle[radius=1pt];
        \draw (-1/2+2,{sqrt(3)/2}) circle[radius=1pt]; and \fill[blue] (-1/2+2,{sqrt(3)/2}) circle[radius=1pt];
        \draw (1/2+2,{sqrt(3)/2}) circle[radius=1pt]; and \fill[blue] (1/2+2,{sqrt(3)/2}) circle[radius=1pt];
        \draw (1+2,0) circle[radius=1pt]; and \fill[blue] (1+2,0) circle[radius=1pt];

        \draw[line width=.5mm,red] (3/4,{sqrt(3)/4}) -- (-1/4,{sqrt(3)/4});
        \draw[line width=.5mm,red] (1/2+2,0) -- (0+2,{sqrt(3)/2});
        
    \end{tikzpicture}
    \caption{The once-punctured torus $\Sigma_{1,1}$ with a punctured removed (blue). The $0/1$ curve is shown in the middle. The $1/0$ curve is shown on the right.}
    \label{fig:OncePuncturedTorus}
\end{figure}
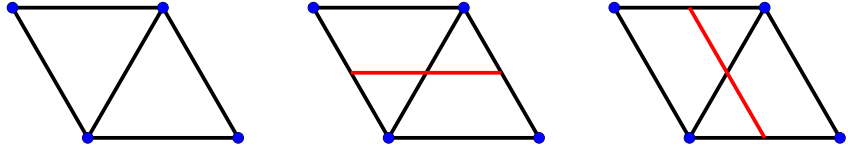

A line in $\mathbb{R}^2$ with slope $p/q$ and disjoint from $\mathbb{Z}^2$ projects to an essential simple closed curve in $\Sigma_{1,1}$. Conversely, an essential simple closed curve in $\Sigma_{1,1}$ lifts to a line in $\mathbb{R}^2$ with slope $p/q$ and disjoint from $\mathbb{Z}^2$. We see that the isotopy classes of essential simple closed curve in $\Sigma_{1,1}$ correspond to $\mathbb{Q}\cup\{1/0\}$. They are organized by the Farey tessalation of $\bbH^2$, see \cref{fig:FareyGraph}. In particular, the ideal vertices of the Farey triangulation coincide with $\mathbb{Q}\cup \{\infty\}$. The edges of the Farey triangulation connecting $p/q$ and $r/s$ if and only if the corresponding simple close curves have geometric intersection number 1. 

\begin{figure}[h!]
    \centering
    \begin{tikzpicture}[scale=2]
    \draw[line width = 1] (0,0) circle (1);
    \draw[line width = 1] (0,1) arc (180:240:{sqrt(3)});
    \draw[line width = 1] ({sqrt(3)/2},-1/2) arc (60:120:{sqrt(3)});
    \draw[line width = 1] (-{sqrt(3)/2},-1/2) arc (-60:0:{sqrt(3)});
    
    \draw[line width = 1] (0,1) arc (180:300:{1/sqrt(3)});
    \draw[line width = 1] (0,1) arc (180:60:{-1/sqrt(3)});

    \draw[line width = 1] ({sqrt(3)/2},-1/2) arc (240:120:{1/sqrt(3)});
    \draw[line width = 1] ({sqrt(3)/2},-1/2) arc (60:180:{1/sqrt(3)});

    \draw[line width = 1] ({-sqrt(3)/2},-1/2) arc (120:0:{1/sqrt(3)});
    \draw[line width = 1] ({-sqrt(3)/2},-1/2) arc (-60:60:{1/sqrt(3)});
    \filldraw[black] ({sqrt(3)/2},-1/2) circle (1pt) node[anchor=west]{$1/1$};
    \filldraw[black] ({sqrt(3)/2},1/2) circle (1pt) node[anchor=west]{$2/1$};
    \filldraw[black] (0,1) circle (1pt) node[above]{$1/0$};
    \filldraw[black] ({-sqrt(3)/2},1/2) circle (1pt) node[left]{$-1/1$};
    \filldraw[black] ({-sqrt(3)/2},-1/2) circle (1pt) node[left]{$0/1$};
    \filldraw[black] (0,-1) circle (1pt) node[below]{$1/2$};
\end{tikzpicture}
    \caption{The Farey graph parametrizing essential simple closed curves on $\Sigma_{1,1}$}
    \label{fig:FareyGraph}
\end{figure}
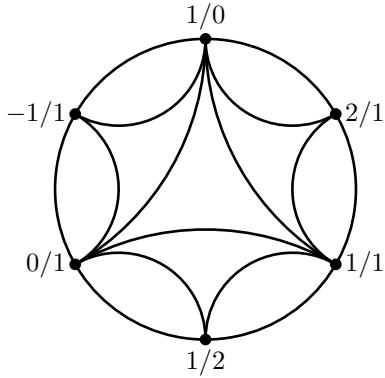

We can parametrize isotopy classes of oriented essential simple closed curve in $\Sigma_{1,1}$ by the following set of vectors: 
\[
\mathcal{U} = \left\{
\begin{pmatrix}
p\\q 
\end{pmatrix}\mid \gcd(p,q) =1, p = \pm 1 \text{ if } q=0, q = \pm 1 \text{ if } p=0 
\right\},
\]
the set of rational direction in $\mathbb{R}^2$. Note that the vector $\begin{pmatrix}
    1 \\ 0
\end{pmatrix}$ corresponds to the positive $y$-direction of $\mathbb{R}^2$ while the vector $\begin{pmatrix}
    0 \\ 1
\end{pmatrix}$ corresponds to the positive $x$-direction of $\mathbb{R}^2$. By abusing notation, we will use elements of $\mathcal{U}$ to denote isotopy classes of oriented essential simple closed curve in $\Sigma_{1,1}$. Similarly, we will use elements of $\mathbb{Q}\cup \{1/0\}$ to the denote the unoriented counterpart in $\Sigma_{1,1}$. 

Since the deck group of $\overline{\pi}:\Sigma_{1,1}\to \Sigma_{\Mod}$ acts by isometries on $\Sigma_{1,1}$, we have an associated $6$-fold cyclic covering $\pi: \UT(\Sigma_{1,1}) \to \UT(\Sigma_{\Mod})$. The unit tangent bundle $\UT(\Sigma_{1,1})$ can be trivialized as a product $\Sigma_{1,1}\times S^1$ where $S^1 = \mathbb{R}/2\pi\mathbb{Z}$. In particular, the oriented curve $\begin{pmatrix}
    p \\ q 
\end{pmatrix}\in\mathcal{U}$ on $\Sigma_{1,1}$ determines a canonical lift to the oriented curve \[\begin{pmatrix}
    p \\ q
\end{pmatrix}\times\left\{\arg \begin{pmatrix}
    p \\ q
\end{pmatrix}\right\}\] 
where $\arg:\mathcal{U} \to [0,2\pi)$ is the angle from $\begin{pmatrix}
    0 \\ 1
\end{pmatrix}$ to $\begin{pmatrix}
    p \\ q
\end{pmatrix}$ in the counter clockwise direction. Since an oriented curve in $\mathcal{U} \subset \Sigma_{1,1}$ completely determines its canonical lift to $\UT(\Sigma_{1,1})$, we also use elements in $\mathcal{U}$ to denote this canonical lift.

By \cite[Lemma 5.1]{MiguelesPinskyPurcell}, the action of a generator $\nu$ of the deck group of $\overline{\pi}:\Sigma_{1,1} \to \Sigma_{\Mod}$ on the oriented curve is by the order $6$ matrix in $\SL_2(\mathbb{Z})$
\[
\nu = \begin{pmatrix}
    0 & 1 \\ -1 & 1
\end{pmatrix}.
\] 
Since an oriented curve in $\mathcal{U}\subset \Sigma_{1,1}$ determines its canonical lift, we can also denote the action the deck group of $\pi:\UT(\Sigma_{1,1})\to \UT(\Sigma_{\Mod})$ on the set of canonical lifts $\mathcal{U} \subset \UT(\Sigma_{1,1})$ is also by the same matrix $ \nu $.

The following lemma from \cite{MiguelesPinskyPurcell} explains the relationship between canonical lifts of oriented closed geodesic in $\Sigma_{1,1}$ in $\mathcal{U}$ and canonical lifts of oriented closed geodesic in $\Sigma_{\Mod}$. 

\begin{lem}[{\cite[Lemma 5.1]{MiguelesPinskyPurcell}}]
    Suppose that $\overline{\gamma}$ is an oriented closed geodesic in $\Sigma_{\Mod}$ obtained by projecting the simple closed curve $p/q \subset \Sigma_{1,1}$ via the covering map $\overline{\pi}:\Sigma_{1,1}\to\Sigma_{\Mod}$. Then the canonical lift $\gamma \subset \UT(\Sigma_{\Mod})$ has six lifts. These lifts are 
    \[
    \left\{\pm \begin{pmatrix}
        p \\ q
    \end{pmatrix},\pm \begin{pmatrix}
        q \\ q-p
    \end{pmatrix},\pm \begin{pmatrix}
        p-q \\ p
    \end{pmatrix}\right\} \subset \UT(\Sigma_{1,1}).
    \]
\end{lem}

A main result of \cite{MiguelesPinskyPurcell} is the following theorem

\begin{thm}[{\cite[Theorem 4.3, 5.3]{MiguelesPinskyPurcell}}]
\label{thm:MPPMainResult}
     Suppose that $\Delta:= \left\{\begin{pmatrix}
         a_j \\ b_j
     \end{pmatrix} \right\} \subset \UT(\Sigma_{1,1}) $ such that:
     \begin{enumerate}
         \item $|\Delta| < \infty$,
         \item $\Delta$ is invariant under the action of $\displaystyle\nu = \begin{pmatrix}
             0 & 1 \\ -1 & 1
         \end{pmatrix}$, and
         \item For every $\begin{pmatrix}
         a_j \\ b_j
     \end{pmatrix}$, there exists $\begin{pmatrix}
         a_i \\ b_i
     \end{pmatrix}$  and $\begin{pmatrix}
         a_k \\ b_k
     \end{pmatrix}$ such that $\left|\det\begin{pmatrix}
         a_i & a_j \\ b_i & b_j
     \end{pmatrix}\right| = \left|\det\begin{pmatrix}
         a_j & a_k \\ b_j & b_k
     \end{pmatrix}\right|  = 1$.
     \end{enumerate}
     Then the manifolds $\UT(\Sigma_{1,1})\setminus \Delta$ and $\UT(\Sigma_{\Mod}) \setminus \pi(\Delta)$ are both arithmetic. 
\end{thm}


Let us denote by $\mathcal{F}\Sigma_{1,1}$ the collection of $\Delta$ where $\Delta$ is the union of canonical lifts of oriented closed geodesics in $\Sigma_{1,1}$ to $\UT(\Sigma_{1,1})$ satisfying the conditions of \cref{thm:MPPMainResult}. The collection of arithmetic modular links that was found in \cite{MiguelesPinskyPurcell} is described as 
\[
\mathcal{F}\Sigma_{\Mod} := \{\Gamma \subset \UT(\Sigma_{\Mod}) \mid  \pi^{-1}(\Gamma) \in \mathcal{F}\Sigma_{1,1}\}
\]

We end with the following observation from \cite{MiguelesPinskyPurcell} underpinning their construction:

\begin{lem}[{\cite[Lemma 4.1]{MiguelesPinskyPurcell}}]
\label{lem:Nalphabeta}
    Let $N_{\alpha,\beta}$ be the manifold 
    \[
    N_{\alpha,\beta} := (\Sigma_{1,1} \times [0,1]) \setminus \{\alpha \times \{0\} \cup \beta \times \{1\}\}
    \]
    where $\alpha$ and $\beta$ are $p/q$ and $r/s$ curves on $\Sigma_{1,1}$ such that $|ps - qr| =1$. Then $N_{\alpha,\beta}$ is homeomorphic to $N_{0,1}$.
\end{lem}

\begin{remark}
\label{rem:PropertyOfHomeoBetweenPieces}
The homeomorphism between $N_{\alpha,\beta}$ and $N_{0,1}$ is induced by the linear map that sends $\alpha$ to $0$ and $\beta$ to $1$. If we orient all the curves $\alpha,\beta,0$ and $1$, then there exists a unique linear transformation that preserves the orientations of the curves and induces the homeomorphism between $N_{\alpha,\beta}$, $N_{0,1}$.
\end{remark}

\section{Proof of the main results}

\subsection{Proof of \cref{thm:ComplementArithmeticModularLinks}}
In this section, we give a proof of \cref{thm:ComplementArithmeticModularLinks}. We begin with the following observation.

\begin{lem}
\label{lem:Delta0}
    For any $\Delta \in \mathcal{F}\Sigma_{1,1}$, $\Delta $ contains 
    \[
    \Delta_0 := \left\{\pm\begin{pmatrix}
        0 \\ 1
    \end{pmatrix},\pm\begin{pmatrix}
        1 \\ 1
    \end{pmatrix},\pm\begin{pmatrix}
        1 \\ 0
    \end{pmatrix}\right\}.
    \]
    Consequently, $\Delta_0$ is the smallest collection in $\mathcal{F}\Sigma_{1,1}$ ordered by inclusion.
\end{lem}

\begin{proof}
We project $\Delta$ to $\Sigma_{1,1}$ to get a collection of essential simple closed curves $\overline{\Delta} \subset \Sigma_{1,1}$. The fact that $\Delta $ is $\nu$-invariant implies that $\overline{\Delta}$ is $V$-invariant where we view $\overline{\Delta} \subset \mathbb{Q}\cup\{\infty\}$. Since $\overline{\Delta}$ is $V$-invariant, $|\overline{\Delta}| = 3n$ for some $n\geq 1$. Furthermore, there are exactly $n$ curves represented by vertices of the Farey graph in the intervals from $0/1$ to $1/1$, from $1/1$ to $1/0$ and from $1/0$ to $0/1$ all oriented counter clockwise. The third condition for $\Delta$ is satisfied only if $\{0/1,1/1,1/0\} \subseteq \overline{\Delta}$. Lifting these curves to $\UT(\Sigma_{1,1})$, we get the desired conclusion for $\Delta$.
\end{proof}

 Let $\Gamma_{0} = \left\{\pi\begin{pmatrix}
        0 \\1 
    \end{pmatrix}\right\}$. Up to a reparametrization, the manifold $M_{\Gamma_{0}}$ is  
    \[
    M_{\Gamma_0} = \frac{\Sigma_{1,1}\times [0,1] \setminus\{0/1\times \{0\} \cup 1/1 \times\{1\}\}}{(x,0)\sim(\nu(x),1)}.
    \]
    Let $\phi:M_{\Gamma_0} \to S^1$ be a surjection coming from projecting onto the second factor which induces a surjective homomorphism $\phi_*:\pi_1(M_{\Gamma_0})\to\mathbb{Z}$. The map $\phi_*$ sends the meridian of the trefoil to $1$ (up to taking inverse) and the meridian of the $0/1$ geodesic to $0$. Let $M_n$ be the cover of $M_{\Gamma_0}$ that corresponds to $\phi_*^{-1}(n\mathbb{Z})$ for some positive integer $n$, then up to a reparametrization of $S^1$ the manifold $M_n$ is 
    \[
     M_n = \frac{\Sigma_{1,1}\times [0,n] \setminus\{\nu^i(0/1)\times \{i\}\}_{i=0}^n}{(x,0)\sim(\nu^n(x),n)}.
    \]
See \cref{fig:ModularLinkExample} for an example of $M_2$. 

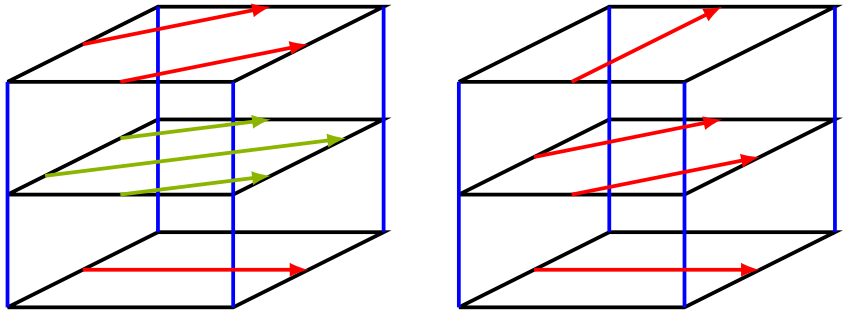
\begin{figure}[h!]
    \centering
    \begin{tikzpicture}[scale=2]
\draw[line width=.5mm] (-3,0) -- (3/2-3,0) -- (5/2-3,1/2) -- (1-3,1/2) -- (0-3,0);
\draw[line width=.5mm] (0-3,3/2) -- (3/2-3,3/2) -- (5/2-3,2) -- (1-3,2) -- (0-3,3/2);
\draw[line width=.5mm] (0-3,3/4) -- (3/2-3,3/4) -- (5/2-3,5/4) -- (1-3,5/4) -- (0-3,3/4);

\draw[line width=.5mm,blue] (0-3,0) -- (0-3,3/2);
\draw[line width=.5mm,blue] (3/2-3,0) -- (3/2-3,3/2);
\draw[line width=.5mm,blue] (5/2-3,1/2) -- (5/2-3,2);
\draw[line width=.5mm,blue] (1-3,1/2) -- (1-3,2);

\draw[line width=.5mm,red,->] (1/2-3,1/4) -- (2-3,1/4); 
\draw[line width=.5mm,red,->] (3/4-3,3/2) -- (2-3,7/4);
\draw[line width=.5mm,red,->] (1/2-3,7/4) -- (7/4-3,2);

\node[above,red] at (-2.5,1/4) {$0/1$}; 
\node[below,red] at (-2.5,2.1) {$1/1$}; 
\node[below,green] at (-2.5,1.3) {$1/2$}; 

\draw[line width=.5mm,green,->] (3/4-3,9/8) -- (7/4-3,5/4);
\draw[line width=.5mm,green,->] (1/4-3,7/8) -- (9/4-3,9/8);
\draw[line width=.5mm,green,->] (3/4-3,3/4) -- (7/4-3,7/8);
\draw[line width=.5mm] (0,0) -- (3/2,0) -- (5/2,1/2) -- (1,1/2) -- (0,0);
\draw[line width=.5mm] (0,3/2) -- (3/2,3/2) -- (5/2,2) -- (1,2) -- (0,3/2);
\draw[line width=.5mm] (0,3/4) -- (3/2,3/4) -- (5/2,5/4) -- (1,5/4) -- (0,3/4);

\draw[line width=.5mm,blue] (0,0) -- (0,3/2);
\draw[line width=.5mm,blue] (3/2,0) -- (3/2,3/2);
\draw[line width=.5mm,blue] (5/2,1/2) -- (5/2,2);
\draw[line width=.5mm,blue] (1,1/2) -- (1,2);

\draw[line width=.5mm,red,->] (1/2,1/4) -- (2,1/4);
\draw[line width=.5mm,red,->] (3/4,3/2) -- (7/4,2);

\node[above,red] at (0.5,1/4) {$0/1$}; 
\node[below,red] at (0.5,2.1) {$1/0$}; 
\node[below,red] at (0.5,1.3) {$1/1$};

\draw[line width=.5mm,red,->] (3/4,3/4) -- (2,1);
\draw[line width=.5mm,red,->] (1/2,1) -- (7/4,5/4);
\end{tikzpicture}
    \caption{On the left is the manifold $M_\Gamma$ where $\Gamma = \{\pi\begin{pmatrix}
        0 \\ 1
    \end{pmatrix},\pi\begin{pmatrix}
        1 \\ 2
    \end{pmatrix}\}$ realized as a once-punctured torus bundle with some curves in red and green drilled out. The gluing map on the left is given by the matrix $\displaystyle \nu = \begin{pmatrix}
        0 & 1 \\ -1 & 1
    \end{pmatrix}$ which glues the bottom to the top. On the right is the manifold $M_2$, the cover of $M_{\Gamma_0}$ corresponds to $\phi_*^{-1}(2\mathbb{Z})$. The gluing map on the right is given by the matrix $\displaystyle \nu^2 = \begin{pmatrix}
        -1 & 1 \\ -1 & 0
    \end{pmatrix}$ which glues the bottom to the top.}
    \label{fig:ModularLinkExample} 
\end{figure}

    
\begin{lem}
\label{lem:HomeoToAugChainLink}
    The manifold $M_n$ is homeomorphic to the complement of the $n$-component augmented chainlink $S^3 \setminus C_n$.
\end{lem}

\begin{proof}
The manifold $M_{\Gamma_0}$ is homeomorphic to the complement of the Whitehead link by a homeomorphism $h:M_{\Gamma_0} \to S^3 \setminus C_1$ decribed in \cite[Figure 8]{MiguelesPinskyPurcell}. Following \cite[Figure 8]{MiguelesPinskyPurcell}, the homeomorphism $h$ is obtained by performing a Rolfsen twist about the component $\Gamma_0$. The homeomorphism $h$ takes the meridian of the trefoil component in the link $\Gamma_0 \cup \{T_{2,3}\}$ to the meridian of a component of the Whitehead link. Therefore, the homeomorphism $h$ lifts to a homeomorphism between $M_n$ and a $n$-fold cyclic cover branched over $h(N(T_{2,3}))$ where $N(T_{2,3})$ is a neighborhood of the trefoil knot. Since the two components of the Whitehead link are symmetric, the latter manifold is $S^3\setminus C_n$.    
\end{proof}


\begin{remark}
    The manifold $M_n$ can be thought of as the complement of a link in the $n$ -fold cyclic covering of the trefoil complement. In general, cyclic coverings of the trefoil complement do not embed into $S^3$. It is surprising that after drilling out some link components they do always embed in $S^3$.
\end{remark}
    

\begin{proof}[Proof of \cref{thm:ComplementArithmeticModularLinks}]
    Let $\Gamma\subset \UT(\Sigma_{\Mod})$ be any modular link in $\mathcal{F}\Sigma_{\Mod}$, $n = |\Gamma|$ and $M_\Gamma = \UT(\Sigma_{\Mod})\setminus \Gamma$. Given \cref{lem:HomeoToAugChainLink}, our goal is to show that $M_\Gamma$ and $M_n$ are homeomorphic. We lift $\Gamma$ to obtain a collection $\Delta \subset \UT(\Sigma_{1,1})$ that is $\nu$-invariant. By \cref{lem:Delta0}, $\Delta$ contains $\begin{pmatrix}
        0 \\ 1
    \end{pmatrix}$ and  $\begin{pmatrix}
        1 \\ 1
    \end{pmatrix}$.
    Up to a reparametrization of $S^1$,  $M_\Gamma$ is  
    \[
    M_\Gamma = \frac{(\Sigma_{1,1} \times [0,n]) \setminus \{\gamma_i \times \{i\}  \}_{i=0}^n}{(x,0)\sim(\nu(x),n)}
    \]
    for $0\leq i \leq n$ where $\gamma_0$ and $\gamma_n$ are $0/1$ and $1/1$ curve on $\Sigma_{1,1}$. Cutting both manifolds $M_\Gamma$ and $M_n$ along thrice-punctured sphere $\Sigma_{1,1} \times \{0\} \setminus \{0/1 \times \{0\}\}$, we get 
    \begin{equation*}
        P_\Gamma = \Sigma_{1,1}\times[0,n] \setminus  \{\gamma_i \times \{i\}  \}_{i=0}^n \text{ and } P_n = \Sigma_{1,1}\times[0,n] \setminus  \{\nu^i(0/1) \times \{i\}  \}_{i=0}^n.
    \end{equation*}
    By \cref{lem:Nalphabeta}, for each $0\leq i \leq n-1$ we have a homeomorphism
    \[
    h_i : N_{\gamma_i,\gamma_{i+1}} \to N_{\nu^i(0),\nu^{i+1}(0)}.
    \]
    By \cref{rem:PropertyOfHomeoBetweenPieces}, we can choose $h_i$ so that they are induced by linear maps that preserve the orientations of the removed curves. Note that $\Sigma_{1,1} \times \{i+1\}\setminus \{\gamma_{i+1}\times \{i+1\}\}$ is homeomorphic to a thrice-punctured sphere. Our choice of $h_i$ ensures that the composition $h_{i-1}^{-1}\circ h_i$ on $\Sigma_{1,1} \times \{i+1\}\setminus \{\gamma_{i+1}\times \{i+1\}\}$ is a homeomorphism of the thrice-punctured sphere that preserves the punctures. Up to isotopy, we can glue the homeomorphism $h_i$'s together and get a homeomorphism $h: P_\Gamma \to P_n$. Note that $h$ is the identity on $\Sigma_{1,1} \times \{0\}\setminus \{0/1 \times \{0\}\}$. Gluing the bottom of $P_\Gamma$ to the top, we get a homeomorphism
    \[
    h: M_\Gamma \to \frac{\Sigma_{1,1}\times[0,n] \setminus  \{\nu^i(0/1) \times \{i\}  \}_{i=0}^n}{(x,0)\sim((h_{n-1}\circ\nu)(x),n)}. 
    \]
    The manifolds $M_n$ and $h(M_\Gamma)$ are obtained from $P_n$ by gluing the bottom to the top via the two maps $\nu^n$ and $h_{n-1}\circ \nu$ respectively. The two gluing maps differ on $\Sigma_{1,1} \times \{0\} \setminus \{0/1 \times \{0\}\}$ by $\nu^{-n}\circ h_{n-1}\circ\nu$. We will show that this homeomorphism is isotopic to the identity. Since $\Sigma_{1,1} \times \{0\} \setminus \{0/1 \times \{0\}\}$ is a thrice-punctured sphere, it suffices to show that the map $\nu^{-n}\circ h_{n-1}\circ\nu$ preserves the punctures of $\Sigma_{1,1} \times \{0\} \setminus \{0/1 \times \{0\}\}$. These punctures comprises of the puncture of $\Sigma_{1,1} \times \{0\}$ and the two sides of the removed geodesic $\{0/1 \times \{0\}\}$. The homeomophism $\nu^{-n}\circ h_{n-1}\circ\nu$ preserves the puncture coming from $\Sigma_{1,1}\times \{0\}$. The homeomorphism $\nu$ up to isotopy is an orientation preserving linear map on $\Sigma_{1,1}$. In particular, $\nu$ preserves the orientation of any oriented simple closed curve. By \cref{rem:PropertyOfHomeoBetweenPieces}, the homeomorphisms $h_i$ can be chosen to preserves the removed geodesic as an oriented curve on $\Sigma_{1,1}$.  Therefore, the map $\nu^{-n}\circ h_{n-1}\circ\nu$ is isotopic to the identity. Therefore, $h(M_\Gamma)$ is homeomorphic to $M_n$.   
\end{proof}


\subsection{Proof of \cref{cor:MostArithmeticModularLinksAreLarge}}

The claim about containing a closed embedded essential surface follows from the work of Cooper and Long \cite{CooperLong93}. For completeness, we give a brief summary of their article focusing on the pertinent details. In this article, the authors studied pure braids from the representation-theoretic and the geometric perspective. On the representation-theoretic side, they introduced and studied the derivative variety associated to an element of the pure braid group \cite[Section 2,3]{CooperLong93}.   

On the geometric side, they studied the complement $S^3 \setminus \hat{\sigma}$ of the closure of a braid $\sigma \in B_n$ \cite[Section 4]{CooperLong93}. In particular, they showed that $S^3 \setminus \hat{\sigma}$ contains a closed essential surface where $\sigma$ is a pure 4-braid lying in the kernel of the Grassner representation \cite[Theorem 4.8]{CooperLong93}. To establish this result, they give general criteria for a link complement in $S^3$ to contain a closed essential surface \cite[Theorem 4.1 and Corollary 4.6]{CooperLong93}. In Theorem 4.1, they showed that if the $\SL_2(\mathbb{C})$-representation variety of an $n$-component link $L \subset S^3$ contains a component of dimension $> n+3$ and has an irreducible representation, then $S^3 \setminus L$ contains a closed essential surface. They pointed out a sufficient condition for the hypothesis of Theorem 4.1 is that the link group $\pi_1(S^3\setminus L)$ surjects a non-abelian free group of rank $k$ such that $3k > n+3$. This is the essential point of \cite[Corollary 4.4]{CooperLong93}. The surjection $\pi_1(S^3\setminus L) \to F_k$ where $F_k$ is the free group of rank $k$ allows one to embed the representation variety of $F_k$ into that of $\pi_1(S^3 \setminus L)$ by pullbacks. The representation variety of the nonabelian free group of rank $k$ contains a component with an irreducible representation and has dimension $3k > n+3$. It follows that the hypothesis of theorem 4.1 is satisfied if $\pi_1(S^3\setminus L)$ surjects a nonabelian free group of rank $k$ such that $3k > n+3$. 

\begin{remark}
    Note that though Corollary 4.4 of \cite{CooperLong93} is stated as removing one component of the link. To apply their argument, one just need the fact that the fundamental group of the link complement surjects a nonabelian free group of sufficiently large rank.  
\end{remark}

\begin{proof}[Proof of \cref{cor:MostArithmeticModularLinksAreLarge}]
        \cref{thm:ComplementArithmeticModularLinks} shows that $M_\Gamma$ is a $|\Gamma|$-fold cyclic cover of $M_{\Gamma_0}$. Since the complement of the Whitehead link fibers, $M_\Gamma$ also fibers. 
        
        The claim about containing a closed embedded essential surface follows from the work of Cooper and Long \cite{CooperLong93}. Suppose that $n = |\Gamma| > 1$, then we write $n =2k$ or $n=2k+1$ where $k\geq 1$ is an integer. The group $\pi_1(S^3\setminus C_n)$ has a surjection onto the free group of rank $k+1$ coming from deleting $k$ components when $n=2k$ and $k+1$ components when $n=2k+1$, see \cref{fig:WHL-AugmentedChainLink}. Therefore, $\pi_1(M_\Gamma)$ surjects a free group of rank $k+1$ for an appropriate $k$. Similar to \cite[Corollary 4.4]{CooperLong93}, the surjection shows that there exists a component of characters of irreducible $\SL_2(\mathbb{C})$-representations of dimension $3k$. The number of cusp of $S^3 \setminus C_n$ is $|\Gamma|+1 = n +1$. When $n=2k$, $3k$ is strictly greater than $n+1$ if and only if $k>1$. When $n =2k+1$, $3k$ is strictly greater than $n+1$ if and only if $k>2$. That is, if $|\Gamma|\not\in\{1,2,3,5\}$, then $S^3\setminus C_n$ satisfy the hypothesis of \cite[Theorem 4.1]{CooperLong93}. It follows from \cite[Theorem 4.1]{CooperLong93} that if $|\Gamma|\not\in\{1,2,3,5\}$, then $S^3\setminus C_n$, and hence $M_\Gamma$, contains a closed embedded essential surface.                        
        \end{proof}

\printbibliography

@misc{MiguelesPinskyPurcell,
      title={Arithmetic modular links}, 
      author={José Andrés Rodríguez Migueles and Tali Pinsky and Jessica S. Purcell},
      year={2023},
      eprint={2307.09409},
      archivePrefix={arXiv},
      primaryClass={math.GT}
}

@incollection {GhysICM06,
    AUTHOR = {Ghys, \'{E}tienne},
     TITLE = {Knots and dynamics},
 BOOKTITLE = {International {C}ongress of {M}athematicians. {V}ol. {I}},
     PAGES = {247--277},
 PUBLISHER = {Eur. Math. Soc., Z\"{u}rich},
      YEAR = {2007},
      ISBN = {978-3-03719-022-7},
   MRCLASS = {37-02 (37C27 37D40 37E99 37J05 53D35 57M25)},
  MRNUMBER = {2334193},
MRREVIEWER = {Michael\ J.\ Usher},
       DOI = {10.4171/022-1/11},
       URL = {https://doi.org/10.4171/022-1/11},
}

@article {MatsusakaUeki23,
    AUTHOR = {Matsusaka, Toshiki and Ueki, Jun},
     TITLE = {Modular knots, automorphic forms, and the {R}ademacher symbols
              for triangle groups},
   JOURNAL = {Res. Math. Sci.},
  FJOURNAL = {Research in the Mathematical Sciences},
    VOLUME = {10},
      YEAR = {2023},
    NUMBER = {1},
     PAGES = {Paper No. 4, 35},
      ISSN = {2522-0144,2197-9847},
   MRCLASS = {57M10 (11F20 11F37 57K10)},
  MRNUMBER = {4519799},
       DOI = {10.1007/s40687-022-00366-8},
       URL = {https://doi.org/10.1007/s40687-022-00366-8},
}

@article {BirmanWilliams,
    AUTHOR = {Birman, Joan S. and Williams, R. F.},
     TITLE = {Knotted periodic orbits in dynamical systems. {I}. {L}orenz's
              equations},
   JOURNAL = {Topology},
  FJOURNAL = {Topology. An International Journal of Mathematics},
    VOLUME = {22},
      YEAR = {1983},
    NUMBER = {1},
     PAGES = {47--82},
      ISSN = {0040-9383},
   MRCLASS = {58F13 (57M25)},
  MRNUMBER = {682059},
MRREVIEWER = {Hans\ G.\ Bothe},
       DOI = {10.1016/0040-9383(83)90045-9},
       URL = {https://doi.org/10.1016/0040-9383(83)90045-9},
}

@article {Dehornoy15,
    AUTHOR = {Dehornoy, Pierre},
     TITLE = {Geodesic flow, left-handedness and templates},
   JOURNAL = {Algebr. Geom. Topol.},
  FJOURNAL = {Algebraic \& Geometric Topology},
    VOLUME = {15},
      YEAR = {2015},
    NUMBER = {3},
     PAGES = {1525--1597},
      ISSN = {1472-2747,1472-2739},
   MRCLASS = {37D40 (37B50 37D45 53D25 57M20)},
  MRNUMBER = {3361144},
MRREVIEWER = {Rafael\ Oswaldo\ Ruggiero},
       DOI = {10.2140/agt.2015.15.1525},
       URL = {https://doi.org/10.2140/agt.2015.15.1525},
}

@book {MilnorKTheory,
    AUTHOR = {Milnor, John},
     TITLE = {Introduction to algebraic {$K$}-theory},
    SERIES = {Annals of Mathematics Studies},
    VOLUME = {No. 72},
 PUBLISHER = {Princeton University Press, Princeton, NJ; University of Tokyo
              Press, Tokyo},
      YEAR = {1971},
     PAGES = {xiii+184},
   MRCLASS = {18F25 (12A65)},
  MRNUMBER = {349811},
MRREVIEWER = {J.-P.\ Jouanolou},
}

@article {FoulonHasselBlatt13,
    AUTHOR = {Foulon, Patrick and Hasselblatt, Boris},
     TITLE = {Contact {A}nosov flows on hyperbolic 3-manifolds},
   JOURNAL = {Geom. Topol.},
  FJOURNAL = {Geometry \& Topology},
    VOLUME = {17},
      YEAR = {2013},
    NUMBER = {2},
     PAGES = {1225--1252},
      ISSN = {1465-3060,1364-0380},
   MRCLASS = {37D20},
  MRNUMBER = {3070525},
MRREVIEWER = {Meirong\ Zhang},
       DOI = {10.2140/gt.2013.17.1225},
       URL = {https://doi.org/10.2140/gt.2013.17.1225},
}

@article {BergeronPinskySilberman,
    AUTHOR = {Bergeron, Maxime and Pinsky, Tali and Silberman, Lior},
     TITLE = {An upper bound for the volumes of complements of periodic
              geodesics},
   JOURNAL = {Int. Math. Res. Not. IMRN},
  FJOURNAL = {International Mathematics Research Notices. IMRN},
      YEAR = {2019},
    NUMBER = {15},
     PAGES = {4707--4729},
      ISSN = {1073-7928,1687-0247},
   MRCLASS = {53C22 (58A32)},
  MRNUMBER = {3988666},
MRREVIEWER = {Paolo\ Piccione},
       DOI = {10.1093/imrn/rnx231},
       URL = {https://doi.org/10.1093/imrn/rnx231},
}

@article {CremaschiMiguelesYarmola,
    AUTHOR = {Cremaschi, Tommaso and Rodrigu\'{r}z-Migueles, Jos\'{e}
              Andr\'{e}s and Yarmola, Andrew},
     TITLE = {On volumes and filling collections of multicurves},
   JOURNAL = {J. Topol.},
  FJOURNAL = {Journal of Topology},
    VOLUME = {15},
      YEAR = {2022},
    NUMBER = {3},
     PAGES = {1107--1153},
      ISSN = {1753-8416,1753-8424},
   MRCLASS = {57K32 (57K10 57K30)},
  MRNUMBER = {4442684},
       DOI = {10.1112/topo.12246},
       URL = {https://doi.org/10.1112/topo.12246},
}

@article {CremaschiMigueles,
    AUTHOR = {Cremaschi, Tommaso and Rodr\'{\i}guez-Migueles, Jos\'{e} A.},
     TITLE = {Hyperbolicity of link complements in {S}eifert-fibered spaces},
   JOURNAL = {Algebr. Geom. Topol.},
  FJOURNAL = {Algebraic \& Geometric Topology},
    VOLUME = {20},
      YEAR = {2020},
    NUMBER = {7},
     PAGES = {3561--3588},
      ISSN = {1472-2747,1472-2739},
   MRCLASS = {57M50},
  MRNUMBER = {4194288},
MRREVIEWER = {Colin\ C.\ Adams},
       DOI = {10.2140/agt.2020.20.3561},
       URL = {https://doi.org/10.2140/agt.2020.20.3561},
}

@misc{Migueles,
      title={Periods of continued fractions and volumes of modular knots complements}, 
      author={José Andrés Rodríguez Migueles},
      year={2023},
      eprint={2008.12436},
      archivePrefix={arXiv},
      primaryClass={math.GT}
}

@article {CooperLong93,
    AUTHOR = {Cooper, D. and Long, D. D.},
     TITLE = {Derivative varieties and the pure braid group},
   JOURNAL = {Amer. J. Math.},
  FJOURNAL = {American Journal of Mathematics},
    VOLUME = {115},
      YEAR = {1993},
    NUMBER = {1},
     PAGES = {137--160},
      ISSN = {0002-9327,1080-6377},
   MRCLASS = {57M07 (20F36 57M05)},
  MRNUMBER = {1209237},
MRREVIEWER = {Athanase\ Papadopoulos},
       DOI = {10.2307/2374725},
       URL = {https://doi.org/10.2307/2374725},
}

@article {Baker08,
    AUTHOR = {Baker, Kenneth L.},
     TITLE = {Surgery descriptions and volumes of {B}erge knots. {I}.
              {L}arge volume {B}erge knots},
   JOURNAL = {J. Knot Theory Ramifications},
  FJOURNAL = {Journal of Knot Theory and its Ramifications},
    VOLUME = {17},
      YEAR = {2008},
    NUMBER = {9},
     PAGES = {1077--1097},
      ISSN = {0218-2165,1793-6527},
   MRCLASS = {57M50 (57M25)},
  MRNUMBER = {2457837},
MRREVIEWER = {Mattia\ Mecchia},
       DOI = {10.1142/S0218216508006518},
       URL = {https://doi.org/10.1142/S0218216508006518},
}

@misc{SnapPy,
     author={Culler, Marc and Dunfield, Nathan M. and Goerner,
     Matthias and Weeks, Jeffrey R.},
     title={Snap{P}y, a computer program for studying the geometry and topology of $3$-manifolds},
     howpublished={Available at \url{http://snappy.computop.org} (DD/MM/YYYY)},
}
\end{document}